\documentclass[11pt,a4paper, reqno]{amsart}
\usepackage{amsmath,amssymb}
\usepackage{bm}
\usepackage[dvipdfmx]{graphicx}
\usepackage{float}
\usepackage[T1]{fontenc}
\usepackage{textcomp}
\usepackage{type1cm}
\usepackage{bbm} 
\usepackage{pifont} 
\usepackage{amsthm} 
\usepackage{stmaryrd} 
\usepackage{cite} 
\usepackage{amsthm} 
\usepackage{ascmac} 
\usepackage[all]{xy}
\usepackage{tikz}
 \usetikzlibrary{arrows} 
\usepackage{mathrsfs} 
\usepackage{url}
\usepackage{alltt} 
\usepackage{braket} 
\usepackage{cases} 
\usepackage{multirow}

\usepackage{mathtools}
\mathtoolsset{showonlyrefs,showmanualtags} 

\makeatletter
\@addtoreset{equation}{section}

\makeatother

\theoremstyle{plain}
\newtheorem{dfn}[subsection]{Definition}
\newtheorem{thm}[subsection]{Theorem}
\newtheorem{cor}[subsection]{Corollary}

\newtheorem{conj}[subsection]{Conjecture}
\newtheorem{prp}[subsection]{Proposition}
\newtheorem{lem}[subsection]{Lemma}

\theoremstyle{definition}

\newtheorem*{ac}{Acknowledgements}

\newcommand{\rarrow}{\longrightarrow}

\newcommand{\DS}{\displaystyle}

\newcommand{\powser}[1]{\llbracket #1 \rrbracket}

\renewcommand{\hat}[1]{\widehat{#1}}
\renewcommand{\tilde}[1]{\widetilde{#1}}

\newcommand{\bb}[1]{\mathbb{#1}}
\newcommand{\fr}[1]{\mathfrak{#1}}
\newcommand{\cal}[1]{\mathcal{#1}}
\newcommand{\scr}[1]{\mathscr{#1}}

\DeclareMathOperator{\Hom}{Hom}
\DeclareMathOperator{\Aut}{Aut}

\DeclareMathOperator{\Ker}{Ker}

\let\Im\relax
\DeclareMathOperator{\Im}{Im}
\DeclareMathOperator{\Coker}{Coker}

\DeclareMathOperator{\Gal}{Gal}
\DeclareMathOperator{\Sel}{Sel}

\DeclareMathOperator{\Z}{\mathbb{Z}}
\DeclareMathOperator{\Q}{\mathbb{Q}}

\DeclareMathOperator{\F}{\mathbb{F}}

\DeclareMathOperator{\ch}{char}
\DeclareMathOperator{\Fitt}{Fitt}

\title[]{On the Mazur--Tate refined conjecture for the anticyclotomic tower at inert primes}
\author{Ryota Shii}
\address{Graduate School for Mathematics, Kyushu University, Motooka 744, Nishi-ku Fukuoka 819-0395, Japan}
\email{shii.ryota@gmail.com}
\date{\today}

\begin{document}
\begin{abstract}
  Let $E$ be an elliptic curve defined over $\Q$ with supersingular reduction at $p \geq 5$, and $K$ be an imaginary quadratic field such that $p$ is inert in $K/\Q$.
  In this paper, we prove the analogous of the ``weak'' Mazur--Tate refined conjecture for an anticyclotomic tower over $K$ using the result by A. Burungale--K. B\"{u}y\"{u}kboduk--A. Lei.
\end{abstract}

\maketitle

\section{Introduction}
\noindent
The Birch and Swinnerton-Dyer (BSD) conjecture for an elliptic curve $E$ defined over $\Q$ states the mysterious relation between the critical value of the Hasse--Weil $L$-function for $E$ and the various arithmetic invariants of $E$.
In \cite{Maz-Tat87}, B. Mazur and J. Tate refined this conjecture including information on the Galois action of a finite abelian extension over $\Q$. 
At the same time, Iwasawa theory is one of the strong tools for understanding such a relation between an arithmetic object and an analytic object.
The Mazur--Tate's formulation can be regarded as the refined Iwasawa theory, for example, \cite{Kur02}, \cite{Kim-Kur21}, \cite{Ota18}, etc.
More precisely, the Mazur--Tate ``weak'' refined conjecture claims that the Mazur--Tate element that interpolates the $L$-values of $E$ twisted by Dirichlet characters with the conductor $S$ belongs to the Fitting ideal of the Selmer group of $E$ over the cyclotomic fields $\Q(\zeta_{S})$.
In \cite{Kim-Kur21}, C.-H. Kim and M. Kurihara proved this conjecture for a $p$-power $S$ from cyclotomic Iwasawa main conjecture of $E$ with good reduction at $p$ under other mild assumption.
In this paper, we consider the analogous of the Mazur--Tate conjecture for the anticyclotomic tower of an imaginary quadratic field $K$.
The anticyclotomic Iwasawa theory for the split case is relatively developed, so the analogous result has been already considered by H. Darmon, C.-H. Kim, etc.
However, we focus on the one for the \textit{inert} case.

Let $E$ be an elliptic curve defined over $\Q$ with the conductor $N_{E}$, and $p \geq 5$ be a prime number such that $E$ has good supersingular reduction at $p$.
Let $K$ be an imaginary quadratic field such that the discriminant $D_{K}$ is prime to $pN$.
Assume that $p$ is inert in $K/\Q$ and
\begin{center}
  (cp) the class number $h_{K}$ is not divisible by $p$.
\end{center}
We write $N_{E}=N_{E}^{+}N_{E}^{-}$ where $N_{E}^{+}$ is divisible only by the split primes in $K/\Q$ and $N_{E}^{-}$ is divisible only by the inert primes in $K/\Q$.
Through this paper, suppose that
\begin{center}
  (def) $N_{E}^{-}$ is the square-free product of an odd number of primes. 
\end{center}
The assumption (def) implies that the root number of $E$ over $K$ is $+1$.
Let $K_{\infty}$ denote the anticyclotomic $\Z_{p}$-extension over $K$ and $K_{n}$ the subfield of $K_{\infty}$ with its degree $[K_{n}:K]=p^{n}$ for any positive integer $n$.
We define the commutative group ring $\Lambda_{n} \coloneq \Z_{p}[\Gal(K_{n}/K)]$.
Under the assumption (def), M. Bertolini and H. Darmon defined the Bertolini--Darmon element $\cal{L}_{f, n}$ in $\Lambda$ including information of $L$-values of $E$, and $L_{p}(E/K_{n}) \coloneq \cal{L}_{f, n} \cdot \iota(\cal{L}_{f, n})$.
Here, $f$ is the newform of weight 2 corresponding to the isogeny class of $E$, and $\iota$ is the involution in $\Lambda_{n}$ sending $\gamma \longmapsto \gamma^{-1}$.
From the view of the statement of Iwasawa main conjecture, C.-H. Kim states the anticyclotomic analogy of the ``weak'' Mazur--Tate refined conjecture in \cite{Kim18+}.
\begin{conj}[``weak'' main conjecture]
  $L_{p}(E/K_{n}) \in \Fitt_{\Lambda_{n}} \left( \Sel_{p^{\infty}}(E/K_{n})^{\vee} \right)$.
  Here, for a $\Lambda$-module $M$, $\Fitt_{\Lambda_{n}}(M)$ is the Fitting ideal on $\Lambda_{n}$ for $M$, and $M^{\vee}$ is the Pontryagin dual of $M$.
\end{conj}

Under the hypothesis (def), C.-H. Kim \cite{Kim18+} proved the ``weak'' main conjecture for an elliptic curve defined over $\Q$ with ordinary reduction at $p$ under mild assumptions.
Furthermore, C.-H. Kim \cite{Kim24+} also proved one for an elliptic curve defined over $\Q$ with supersingular reduction at $p$ if $p$ is split in an imaginary quadratic field.
The following theorem is the analogous result to \cite{Kim24+} for the inert case.

\begin{figure}[H]
  \centering
  {
    \renewcommand{\arraystretch}{1.25}
    \begin{tabular}{|c||c|c|}
      \hline
      & ordinary at $p$ & supersingular at $p$ \\ \hline
      $p$ is split in $K/\Q$ & \multirow{2}{*}{C.-H. Kim \cite{Kim18+}} & C.-H. Kim \cite{Kim24+} \\
      $p$ is inert in $K/\Q$ & & Our setting \\ \hline
    \end{tabular}
  }
  \caption{Previous research for the definite type}
\end{figure}

\begin{thm}\label{thm:mthm}
  Assume (cp), (def), and the following:
  \begin{itemize}
    \item[(Im)] the mod $p$ Galois representation $G_{\Q} \rarrow \Aut_{\F_{p}}(E[p])$ is surjective if $p=5$ and irreducible if $p> 5$,
    \item[(ram)] for any prime $\ell \mid N^{-}_{E}$ with $\ell^{2} \equiv 1 \bmod p$, the action of inertia subgroup at $\ell$ on $E[p]$ is non-trivial.  
  \end{itemize} 
  Then, we obtain $L_{p}(E/K_{n}) \in \Fitt_{\Lambda_{n}}\left( \Sel_{p^{\infty}}(E/K_{n})^{\vee} \right)$.
\end{thm}

The proof of Theorem \ref{thm:mthm} is the similar arguments to \cite{Kim-Kur21} and \cite{Kim24+}.
In other words, we obtain Theorem \ref{thm:mthm} from a part of the plus/minus anticyclotomic Iwasawa main conjecture for $E$ and the algebraic structure of the plus/minus Selmer group.
Since we treat the inert case, we use the Rubin conjecture resolved by A. Burungale, S. Kobayashi, and K. Ota in \cite{BKO21} for the local condition of the plus/minus Selmer group.

As an arithmetic consequence, Theorem \ref{thm:mthm} implies the ``weak vanishing'' conjecture by \cite[Proposition 3 in Chapter 1]{Maz-Tat87}.
Let $\chi:\Gal(K_{n}/K) \rarrow \overline{\Q}_{p}^{\times}$ be a character and we extend it to $\Lambda_{n}$ linearly.
We write $I_{\chi}$ for the augmentation ideal $\Ker(\chi:\Lambda_{n} \rarrow \overline{\Q}_{p})$ at $\chi$.
Let $(E(K_{n}) \otimes \overline{\Q}_{p})^{\chi}$ be the $\chi$-isotypic subspace of $E(K_{n}) \otimes \overline{\Q}_{p}$.

\begin{cor}
  Under the hypothesis in Theorem \ref{thm:mthm}, we have $L_{p}(E/K_{n}) \in I_{\chi}^{r_{\chi}}/I_{\chi}^{r_{\chi}+1}$, where $r_{\chi} \coloneq \dim_{\overline{\Q}_{p}} (E(K_{n}) \otimes \overline{\Q}_{p})^{\chi} $.
\end{cor}

\begin{ac}
  The author is grateful to Shinichi Kobayashi for reading carefully his manuscript.
  He is thankful to Chan-Ho Kim for suggesting this problem and answering some questions.
  He would like to thank Takenori Kataoka for the discussion.
  The author was supported by JST SPRING, Grant Number JPMJSP2136.
\end{ac}

\section{Bertolini--Darmon elements}\label{sect:L-func}
\noindent
Let $p \geq 5$ be a prime number and $K$ an imaginary quadratic field where $p$ is inert.
In this section, we recall the definition of the Bertolini--Darmon elements of $E$ over the anticyclotomic $\Z_{p}$-extension of $K$. 
In details, see \cite{Ber-Dar05} or \cite{Dar-Iov08}.
As for the formulation using an automorphic representation, see \cite{BHKO24+}.

Fix a newform $f = \sum_{n=1}^{\infty} a_{n}(f)q^{n}$ of weight 2 and level $N$ whose Fourier coefficients lie in $\Q$.
Assume $p \nmid N$ and $a_{p}(f)=0$. 
Let $N$ decompose to $N^{+}$ and $N^{-}$ such that a prime divisor of $N_{E}^{+}$ (resp. $N_{E}^{-}$) is split (resp. inert) in $K/\Q$ respectively.
We assume that $N^{-}$ is square-free and the number of prime factors of $N^{-}$ is odd.  

Let $B$ be the definite quaternion algebra of discriminant $N^{-}\infty$, and let $R$ be an Eichler $\Z[1/p]$-order of conductor $N^{+}$ in $B$.
By the assumption $p \nmid N^{-}$, we may fix an isomorphism 
\begin{align}
  i_{p}:B_{p} \coloneq B \otimes \Q_{p} \rarrow M_{2}(\Q_{p}).
\end{align}
Let $\cal{T}$ denote the Bruhat--Tits tree of $B_{p}^{\times}/\Q_{p}^{\times}$ attached to $\mathrm{PGL}_{2}(\Q_{p})$.
We write $\cal{V}(\cal{T})$ and $\vec{\cal{E}}(\cal{T})$ as the set of vertices and ordered edges of $\cal{T}$, respectively.
Let $\Gamma \coloneq R^{\times}/\Z[1/p]^{\times}$. 
The group $\mathrm{GL}_{2}(\Q_{p})$ acts naturally on $\cal{T}$ and the quotient of $\cal{T}$ by the action of $\Gamma$.
For a ring $Z$, a $Z$-valued function $h$ on $\vec{\cal{E}}(\cal{T})$ is a \textit{$Z$-valued modular form} on $\cal{T}/\Gamma$ with weight 2 if $h$ satisfies $h(\gamma e)=h(e)$ for all $\gamma \in \Gamma$.
Let $S_{2}(\cal{T}/\Gamma, Z)$ be the set of such modular forms. 
Like the classical theory, modular forms on $\cal{T}/\Gamma$ have many nice properties, such as Hecke operators, oldforms, newforms, etc.
(In details, see \cite[Section 1.1]{Ber-Dar05}.)

The following implies that a newform $f$ can be identified with an element $h \in S_{2}(\cal{T}/\Gamma, \Z)$.

\begin{prp}[Jacquet--Langrands correspondence] \label{prp:J-L corr}
  For a newform $f \in S_{2}(\Gamma_{0}(N))$, there exists $h \in S_{2}(\cal{T}/\Gamma, \Z)$ such that it shares the same eigenvalues as $h$ for the Hecke operators $T_{\ell}$ ($\ell \nmid N$).
  This form $h$ is unique up to multiplication by a non-zero scalar.
\end{prp}
\begin{proof}
  See \cite[Proposition 1.3]{Ber-Dar05}.
\end{proof}

From now, $f$ is identified with $h$ corresponding by Proposition \ref{prp:J-L corr}.
Under the above preparations, we define the Bertolini--Darmon elements for $h$.
For simplicity, we assume that the class number of the maximal $\Z[1/p]$-order $\cal{O}_{K}[1/p]$ in $K$ is equal to 1 by the end of this section.
However, a similar result is known in the general case for adelic arguments.
(See \cite{Ber-Dar05}.)
Fix an embedding $\Psi:K \rarrow B$ such that $\Psi(K) \cap R = \Psi(\cal{O}_{K}[1/p])$, that is $\Psi^{-1}(R) = \cal{O}_{K}[1/p]$.
Such an embedding exists uniquely by the definition of $N^{+}$ and $N^{-}$ and the assumption of the class number of $\cal{O}_{K}[1/p]$.
Let $K_{p}$ be the completion field of $K$ at $p$. 
The embedding $K_{p}^{\times} \rarrow B_{p}^{\times}$ induced by $\Psi$ yields an action of $\tilde{G}_{\infty} \coloneq K_{p}^{\times}/\Q_{p}^{\times}$ on $\cal{T}$, denoted by 
\begin{align}
  \sigma \star x \coloneq (i_{p} \circ \Psi)(\sigma)x
\end{align} 
for any $\sigma \in \tilde{G}_{\infty}$ and any vertex or edge $x$ of $\cal{T}$.
Note that $\tilde{G}_{\infty}$ is a compact group because $p$ is inert in $K/\Q$.
Then, there exists $v_{0} \in \cal{V}(\cal{T})$ such that we have $\sigma \star v_{0} = v_{0}$ for all $\sigma \in \tilde{G}_{\infty}$.
Let $U_{0} = \tilde{G}_{\infty} \supset U_{1} \supset U_{2} \supset  \cdots \supset U_{n} \supset \cdots$ be the natural decreasing filtration of the group $U_{0}$ by subgroups of index $(p+1)p^{n-1}$ as in \cite[(2.3)]{Dar-Iov08}.
Fix a sequence $v_{1}, \ldots, v_{n}, \ldots$ of vertices of $\cal{T}$ such that $v_{i}$ and $v_{i+1}$ are connected for all $i \geq 0$.
Let $\tilde{G}_{n} \coloneq \tilde{G}_{\infty}/U_{n}$.
Then, we define $h_{K, n}:\tilde{G}_{n} \rarrow \Z_{p}$ as $h_{K,n}(\sigma) \coloneq h(\sigma \star v_{n})$ and
\begin{align}
  \tilde{\cal{L}}_{h,m} \coloneq \sum_{\sigma \in \tilde{G}_{m}} h_{K, m}(\sigma) \sigma^{-1} \in \Z_{p}[\tilde{G}_{m}].
\end{align}
Note that the stabilizer of $v_{n}$ in $\tilde{G}_{\infty}$ is equal to $U_{n}$ for all $n \geq 0$, and thus $h_{K, n}$ is well-defined.

Let $\pi_{m+1, m}: \Z_{p}[\tilde{G}_{m+1}] \rarrow \Z_{p}[\tilde{G}_{m}]$ be the natural projection, and $\tilde{\xi}_{m}: \Z_{p}[\tilde{G}_{m}] \rarrow \Z_{p}[\tilde{G}_{m+1}]$ the norm map.
Then, we have a recurrence formula $\pi_{m+1, m}(\tilde{\cal{L}}_{h, m+1}) = a_{p}(h)\tilde{\cal{L}}_{h, m} - \tilde{\xi}_{m-1} \tilde{\cal{L}}_{h, m-1}$, in particular, 
\begin{align}
  \pi_{m+1, m}(\tilde{\cal{L}}_{h, m+1}) = - \tilde{\xi}_{m-1} \tilde{\cal{L}}_{h, m-1} \label{eq:recurrence}
\end{align}
by the assumption $a_{p}(h)=0$ (for the proof, see \cite[Lemma 2.6]{Dar-Iov08}).

The group $\tilde{G}_{\infty}$ can be decomposed into $\Delta \times G_{\infty}$, where $\Delta$ is the torsion subgroup of $\tilde{G}_{\infty}$ and $G_{\infty}$ is its maximal torsion-free quotient.
It is known that $G_{\infty}$ is isomorphic to $\Z_{p}$ as a topological group.
By abuse of notation, we also denote the natural projection $\Z_{p}[G_{m+1}] \rarrow \Z_{p}[G_{m}]$ as $\pi_{m+1, m}$.
Let $\cal{L}_{n}$ and $\xi_{n}$ be the projection of $\tilde{\cal{L}}_{n}$ and $\tilde{\xi}_{n}$ by the natural projection $\Z_{p}[\tilde{G}_{m}] \rarrow \Z_{p}[G_{m}]$, respectively. 
Applying the natural projection $\Z_{p}[\tilde{G}_{m}] \rarrow \Z_{p}[G_{m}]$ to the recurrence formula \eqref{eq:recurrence}, we have
\begin{align}
  \pi_{m+1, m}(\cal{L}_{h, m+1}) = - \xi_{m-1} \cal{L}_{h, m-1}
\end{align}
for all $n \geq 1$.
The element $\cal{L}_{h, m}$ is called the $m$-th Bertolini--Darmon elements of $h$.
Note that $G_{\infty}$ is identified with the Galois group of the anticyclotomic $\Z_{p}$-extension $K_{\infty}/K$ by class field theory.
Given this, we define
\begin{align}
  L_{p}(f/K_{n}) \coloneq \cal{L}_{f, n} \cdot \iota(\cal{L}_{f, n}),
\end{align} 
where $K_{n}$ is the subfield of $K_{\infty}$ with $[K_{n}:K] = p^{n}$ and $\iota$ denote the involution sending $\gamma \longmapsto \gamma^{-1}$.

The Bertolini--Darmon elements satisfy the interpolation formula as the following:
Let $\alpha, \beta$ be the roots of the Hecke polynomial of $f$ at $p$.
Fix $\lambda \in \{ \alpha, \beta \}$ and define
\begin{align}
  \cal{L}_{f, n}^{\lambda} \coloneq \frac{1}{\lambda^{n+1}}\left( \cal{L}_{f, n} - \frac{1}{\lambda} \xi_{n-1}(\cal{L}_{f, n-1}) \right).
\end{align}
It is clear that we have $\pi_{n+1, n}(\cal{L}_{f, n+1}^{\lambda})=\cal{L}_{f, n}^{\lambda}$ by the recurrence formula \eqref{eq:recurrence} and $a_{p}(f) = 0$.
Therefore, the sequence $\left( \cal{L}_{f, n}^{\lambda} \right)_{n}$ converges to an element $\cal{L}_{f}^{\lambda}$ in the set of power series in $\Q_{p}\powser{T}$ such that converge in the open unit disk.

\begin{prp}[{\cite{Ber-Dar05}}]
  If $\chi$ is a finite character of $\Gal(K_{\infty}/K)$, then we have
  \begin{align}
    \left( \cal{L}_{f}^{\lambda}\cdot\iota(\cal{L}_{f}^{\lambda}) \right) (\chi) \overset{\cdot}{=} \frac{L(f \otimes \chi, 1)}{\sqrt{|D_{K}|} \Omega_{f}},
  \end{align}
  where $\Omega_{f} \coloneq \langle f, f \rangle_{\mathrm{Pet}}$, and $\overset{\cdot}{=}$ means the equality up to a non-zero algebraic fudge factor.    
\end{prp}

Fix a topological generator $\gamma \in G_{\infty}$.
Let $\phi_{p^{n}}(T)$ denote the $p^{n}$-th cyclotomic polynomial.
Then, we can identify $\xi_{n}$ with $\phi_{p^{n}}(\gamma-1)$ in $\Z_{p}[G_{n}]$.
For a positive integer $n$, we define 
\begin{align}
  \tilde{\omega}_{n}^{+} \coloneq \prod_{\substack{1 \leq k \leq n \\ k:\text{even}}} \xi_{k}, \quad \textrm{and } 
  \tilde{\omega}_{n}^{-} \coloneq \prod_{\substack{1 \leq k \leq n \\ k:\text{odd}}} \xi_{k}.
\end{align}
Furthermore, let $\omega_{0}^{+} \coloneq 1$, $\omega_{0}^{-} \coloneq \gamma -1$, $\omega_{n}^{+} \coloneq \tilde{\omega}_{n}^{+}$, and $\omega_{n}^{-} \coloneq (\gamma - 1) \tilde{\omega}_{n}^{-}$, which are the same notations as \cite{BKOpre}.
The following is a key proposition for the construction of the plus/minus $p$-adic $L$-functions.

\begin{prp}[{\cite[Proposition 2.8]{Dar-Iov08}}] \label{prp:Ln properties}
  We have the following.
  \begin{enumerate}
    \item If $n$ is even, there exists uniquely $L_{n}^{+} \in \Lambda/(\gamma -1)\omega_{n}^{+}\Lambda$ such that $\cal{L}_{h, n}=\tilde{\omega}_{n}^{-}L_{n}^{+}$.
    \item If $n$ is odd, there exists uniquely $L_{n}^{-} \in \Lambda/\omega_{n}^{-}\Lambda$ such that $\cal{L}_{h, n}=\tilde{\omega}_{n}^{+}L_{n}^{-}$.    
  \end{enumerate}
\end{prp}

Now, we define 
\begin{align}
  \cal{L}_{n}^{+} &\coloneq (-1)^{n/2}L_{n}^{+} \quad \text{if $n$ is even}, \\
  \cal{L}_{n}^{-} &\coloneq (-1)^{(n+1)/2}L_{n}^{-} \quad \text{if $n$ is odd}.
\end{align}
Then, the sequences $\{ \cal{L}_{n}^{+} \}_{n: \text{even}}$, $\{ \cal{L}_{n}^{-} \}_{n: \text{odd}}$ are compatible with the natural projections $\Lambda/(\gamma - 1)\omega_{n}^{+} \rarrow  \Lambda/(\gamma -1)\omega_{n-2}^{+}$ and $\Lambda/\omega_{n}^{-} \rarrow  \Lambda/\omega_{n-2}^{-}$, respectively.
Therefore, we can define 
\begin{align}
  \cal{L}_{h}^{+} &\coloneq \varprojlim_{n} \cal{L}_{n}^{+} \in \varprojlim_{n} \Lambda/(\gamma -1)\omega_{n}^{+} \simeq \Lambda, \\
  \cal{L}_{h}^{-} &\coloneq \varprojlim_{n} \cal{L}_{n}^{-} \in \varprojlim_{n} \Lambda/\omega_{n}^{-} \simeq \Lambda.
\end{align}
Then, the plus/minus $p$-adic $L$-functions for $h$ and $K$ are defined by
\begin{align}
  L_{p}^{\pm}(h/K_{\infty}) \coloneq \cal{L}_{h}^{\pm} \cdot \iota \bigl( \cal{L}_{h}^{\pm} \bigr).
\end{align}
For a newform $f$ corresponding to $h \in S_{2}(\cal{T}/\Gamma, \cal{O}_{F})$ by Proposition \ref{prp:J-L corr}, we also define $L_{p}^{\pm}(f/K_{\infty}) \coloneq L_{p}^{\pm}(h/K_{\infty})$ and $L_{p}(f/K_{n}) \coloneq L_{p}(h/K_{n})$, respectively.

\section{The plus/minus Selmer group}
\noindent
In this section, we define the plus/minus Selmer group and introduce previous works used in the proof of Theorem \ref{thm:mthm}.
As in the previous section, fix a prime number $p \geq 5$.

First, we recall the Rubin conjecture. 
(For the detail, see \cite{BKO21}, \cite{BKOpre}.)
Let $\Phi$ be the unramified quadratic extension of $\Q_{p}$ and $\cal{O}$ its ring of integers.
We fix the Lubin--Tate formal group $\mathscr{F}$ defined over $\cal{O}$ of height 2 with the parameter $\pi \coloneq -p$.
For $n \geq -1$, we define $\Phi_{n} = \Phi(\mathscr{F}[\pi^{n+1}])$ by the extension of $\Phi$ generated by the $\pi^{n+1}$-torsion points of $\mathscr{F}$, and put $\Phi_{\infty} = \bigcup_{n \geq 0} \Phi_{n}$.
Then, the Galois action on the $\pi$-adic Tate module $T_{\scr{F}}\coloneq T_{\pi}\scr{F}$ of $\scr{F}$ induces a natural isomorphism
\begin{align}
 \Gal(\Phi_{\infty}/\Phi) \rarrow \Aut(T_{\scr{F}}) \simeq \cal{O}^{\times} \simeq \Delta \times \cal{O}.
\end{align}
by the Lubin--Tate theory, where $\Delta=\Gal(\Phi_{0}/\Phi) \simeq (\cal{O}/\pi\cal{O})^{\times}$.
By the natural action of $\Gal(\Phi/\Q_{p})$ on $G \coloneq \Gal(\Phi_{\infty}/\Phi_{0})$, $G$ has the maximal subgroup $G^{-} \simeq \Z_{p}$ such that $\Gal(\Phi/\Q_{p})$ acts via the non-trivial character.
Let $\Psi_{\infty}$ be the fixed subfield of $\Phi_{\infty}$ by $G^{-}$, and $\Psi_{n}$ be the subfield of $\Psi_{\infty}$ such that $\Gal(\Psi_{n}/\Phi) \simeq \Z/p^{n}\Z$. 
Define the Iwasawa algebra $\Lambda_{\mathcal{O}}\coloneq\mathcal{O} \llbracket G^{-} \rrbracket$.
We fix a topological generator $\gamma$ of $G^{-}$.

We define $T_{\scr{F}}^{\otimes -1} \coloneq \Hom_{\cal{O}}(T_{\scr{F}}, \cal{O})$ and $\bb{T} \coloneq T_{\scr{F}}^{\otimes -1}(1) \otimes_{\cal{O}} \Lambda_{\cal{O}}$, where $T_{\scr{F}}^{\otimes -1}(1)$ is the Tate twist for $T_{\scr{F}}^{\otimes -1}$.
In particular, we have a canonical isomorphism $H^{1}(\Phi, \bb{T}) \simeq \varprojlim_{n} H^{1}(\Psi_{n}, T_{\scr{F}}^{\otimes -1}(1))$ by Shapiro's lemma.
Let $\Xi$ be the set of finite characters of $G^{-}$,
\begin{align}
 \Xi^{+} &\coloneq \{ \chi \in \Xi ~\vline~ \textrm{the conductor of $\chi$ is an even power of $p$} \}, \\
 \textrm{ and } \Xi^{-} &\coloneq \{ \chi \in \Xi ~\vline~ \textrm{the conductor of $\chi$ is an odd power of $p$} \}.
\end{align}
For $\chi \in \Xi$, we identify $\mathrm{coLie}(\scr{F}) \otimes_{\cal{O}} \Phi(\Im(\chi))$ with $\Phi(\Im(\chi))$ by fixing a $\cal{O}$-basis of $\mathrm{coLie}(\scr{F})$.
Then, we define $\exp_{\chi}^{*}:H^{1}(\Phi, \bb{T}) \rarrow \Phi(\Im(\chi))$ by the composition of the map $H^{1}(\Phi, \bb{T}) \rarrow H^{1}(\Phi, T_{\scr{F}}^{\otimes -1}(1)(\chi))$ induced the specialization map $\Lambda_{\cal{O}} \rarrow \cal{O}[\Im (\chi)]$ at $\chi$ and the dual exponential map $H^{1}(\Phi, T_{\scr{F}}^{\otimes -1}(1)(\chi)) \rarrow \mathrm{coLie}(\scr{F}) \otimes_{\cal{O}} \Phi(\Im(\chi)) = \Phi(\Im(\chi))$.
\begin{dfn}
  $H^{1}_{\pm}(\Phi, \bb{T}) \coloneq \{ x \in H^{1}(\Phi, \bb{T}) ~\vline~ \exp_{\chi}^{*}(x) = 0 \textrm{ for all $\chi \in \Xi^{\mp}$} \}$
\end{dfn}

K. Rubin showed that $H^{1}_{\pm}(\Phi, \bb{T})$ is a free $\Lambda_{\cal{O}}$-module of rank 1 (cf. \cite{Rub87}, \cite{BKOpre}).
The following theorem is conjectured by K. Rubin in \cite{Rub87}, and proved by A. Burungale, S. Kobayashi, and K. Ota (cf. \cite{BKO21}).

\begin{thm}[Rubin conjecture]\label{Rubin conj}
  $H^{1}(\Phi, \bb{T}) = H^{1}_{+}(\Phi, \bb{T}) \oplus H^{1}_{-}(\Phi, \bb{T})$.
\end{thm}

Using Rubin conjecture, we define the plus/minus subgroup of local points of $\scr{F}$.
Let $v_{\pm}$ be a $\Lambda$-basis for $H^{1}_{\pm}(\Phi, \bb{T})$, and $v_{\pm, n}$ be the image of $v_{\pm}$ by the map $H^{1}(\Phi, \bb{T}) \rarrow H^{1}(\Psi_{n}, T_{\scr{F}}^{\otimes -1}(1))$.
Define $\Lambda_{\mathcal{O}, n}=\mathcal{O}[\Gal(\Psi_{n}/\Phi)]$ for any $n \geq 0$.
By \cite[Lemma 2.3]{BKOpre}, $v_{\pm, n}$ is a $\Lambda_{\cal{O}, n}$-basis of $H^{1}(\Psi_{n}, T_{\scr{F}}^{\otimes -1}(1))$.
We define the pairing
\begin{align}
  (~,~)_{\Lambda_{\cal{O}, n}}:H^{1}(\Psi_{n}, T_{\mathscr{F}}) \times H^{1}(\Psi_{n}, T_{\mathscr{F}}^{\otimes -1}(1)) \rarrow \Lambda_{\cal{O}, n},
\end{align}
which is extended the natural pairing $H^{1}(\Psi_{n}, T_{\mathscr{F}}) \times H^{1}(\Psi_{n}, T_{\mathscr{F}}^{\otimes -1}(1)) \rarrow \cal{O}$. 
(In details, see \cite[Section 2.2.2]{BKOpre}.) 
The pairing $(~, ~)_{\Lambda_{\cal{O}, n}}$ is perfect by \cite[Lemma 2.3(2)]{BKOpre}. 
Let $v_{\pm, n}^{\perp}$ be the dual basis of $v_{\pm, n}$ with respect to the above pairing $(~, ~)_{\Lambda_{\cal{O}, n}}$.
Using these, we define local points
\begin{align}
  c_{n}^{\pm} \coloneq \omega_{n}^{\mp}v_{\pm, n}^{\perp} \in H^{1}(\Psi_{n}, T_{\mathscr{F}}).
\end{align}
Since we have $c_{n}^{\pm} \in H_{f}^{1}(\Psi_{n}, T_{\mathscr{F}})$ by \cite[Lemma 2.5]{BKOpre}, $c_{n}^{\pm}$ can be seen as elements of $\mathscr{F}(\fr{m}_{n})$.

For any $n \geq 0$, let $\Xi_{n}^{\pm}$ be the set of $\chi \in \Xi^{\pm}$ factoring through $\Gal(\Psi_{n}/\Phi)$.
For $\chi \in \Xi_{n}^{\pm}$ and $x \in \mathscr{F}(\fr{m}_{n})$, we define
\begin{align}
 \lambda_{\chi}(x) = \frac{1}{p^{n}} \sum_{\sigma \in \Gal(\Psi_{n}/\Phi)} \chi^{-1}(\sigma)\lambda(x)^{\sigma},
\end{align}
and the $\Lambda_{\cal{O}, n}$-submodule of $\scr{F}(\fr{m}_{n})$
\begin{align}
 \scr{F}(\fr{m}_{n})^{\pm}=\{ x \in \scr{F}(\fr{m}_{n}) ~\vline~ \lambda_{\chi}(x)=0 \textrm{ for all } \chi \in \Xi_{n}^{\pm} \},
\end{align}
where $\fr{m}_{n}$ is the maximal ideal of $\Psi_{n}$.

\begin{thm}[{\cite[Theorem 2.7]{BKOpre}}]\label{thm:decomposition}
  Let $n \geq 0$.
  Then, we have the following.
  \begin{enumerate}
    \item $c_{n}^{\pm} \in \mathscr{F}(\fr{m}_{n})^{\pm}$, and $\mathscr{F}(\mathfrak{m}_{n})^{\pm}$ are generated by $c_{n}^{\pm}$ as $\Lambda_{\mathcal{O}, n}$-module.
    
    \item $\mathscr{F}(\mathfrak{m}_{n})=\mathscr{F}(\mathfrak{m}_{n})^{+} \oplus \mathscr{F}(\mathfrak{m}_{n})^{-}$ as $\Lambda_{\mathcal{O}, n}$-module.
  \end{enumerate}
\end{thm}

The following are more details of the structure of $\scr{F}(\fr{m}_{n})^{\pm}$ as a $\Lambda_{\cal{O}, n}$-module.
Recall that we define 
\begin{align}
  \omega_{n}^{+}=\omega_{n}^{+}(\gamma)=\prod_{\substack{1 \leq k \leq n \\ k: \textrm{even}}} \phi_{p^{k}}(\gamma), \quad
  \omega_{n}^{-}=\omega_{n}^{-}(\gamma)=(\gamma -1)\prod_{\substack{1 \leq k \leq n \\ k:\mathrm{odd}}} \phi_{p^{k}}(\gamma),
\end{align}
and $\omega_{0}^{+} = 1$, $\omega_{0}^{-} = \gamma -1$ in the previous section.

\begin{prp}\label{prp:n-th structure}
  We have $\scr{F}(\fr{m}_{n})^{\pm} \simeq \omega_{n}^{\mp} \Lambda_{\cal{O}, n} \simeq \Lambda_{\cal{O}, n}/\omega_{n}^{\pm}\Lambda_{\cal{O}, n}$.
  Furthermore, we have 
  \begin{align}
    \Hom_{\cal{O}}(\scr{F}(\fr{m}_{n})^{\pm} \otimes_{\cal{O}} (\Phi/\cal{O}), \Phi/\cal{O}) \simeq \Lambda_{\cal{O}, n}/\omega_{n}^{\pm}\Lambda_{\cal{O}, n}.
  \end{align}
\end{prp}
\begin{proof}
  The first statement follows from $(c_{n}^{\pm}, v_{\mp,n})_{\Lambda_{\cal{O}, n}} = (\omega_{n}^{\mp}v_{\pm, n}^{\perp}, v_{\mp, n})_{\Lambda_{\cal{O}, n}} = \omega_{n}^{\mp}$ (see \cite[(2.14)]{BKOpre}).
  The last statement is proved as same as \cite[Proposition 4.1]{Kim-Kur21}.
\end{proof}

Next, we define the plus/minus Selmer group of $E$ over the anticyclotomic $\Z_{p}$-extension of $K$.
Let $K_{\infty}/K$ be the anticyclotomic $\Z_{p}$-extension as the previous section, and $K_{n}$ the subfield of $K_{\infty}$ such that the Galois group of $K_{n}/K$ is cyclic with the order $p^{n}$.
Let $\cal{O}_{p}$ be the valuation ring of $K_{p}$.
Suppose that 
\begin{center}
  (cp) $p$ is not divisible the class number $h_{K}$ of $K$.
\end{center}
For any non-negative integer $n$, let $K_{n, p}$ be the completion of $K_{n}$ at $p$.
Note that $K_{p}$ is isomorphic to $\Phi$ because $p$ is inert in $K/\Q$, and $K_{n, p}$ is isomorphic to $\Psi_{n}$ by the assumption (cp).
Let $\Lambda \coloneq \Z_{p}\powser{\Gal(K_{\infty}/K)}$.
Note that $\Lambda_{\cal{O}_{p}}$ is a free $\Lambda$-module of rank 2,
and for a $\Lambda_{\cal{O}_{p}}$-module $M$, $\Hom_{\cal{O}_{p}}(M, K_{p}/\cal{O}_{p})$ can be identified with the Pontryagin dual $M^{\vee}$ of $M$ as a $\Lambda$-module.

Let $E$ be an elliptic curve defined over $\Q$ with the conductor $N_{E}$, and $f \in S_{2}(\Gamma_{0}(N_{E}))^{\text{new}}$ be the weight 2 eigenform corresponding to the isogeny class of $E$.
Assume that $E$ has good supersingular reduction at $p$.
We decompose $N_{E}$ into $N_{E}^{+}$ and $N_{E}^{-}$ similarly, and assume that
\begin{center}
  (def) $N_{E}^{-}$ is a square-free product of the odd number of primes.
\end{center}
Let $\hat{E}$ be the formal group of $E$ defined over $\cal{O}_{p}$.
Note that $\hat{E}$ is isomorphism to the Lubin--Tate formal group $\scr{F}$ because $E/\Q$ has a good supersingular reduction at $p \geq 5$.
Let $T=T_{p}E$ be the $p$-adic Tate module of $E$, $V \coloneq T \otimes \Q_{p}$ and $A \coloneq V/T$.

With the above preparations, we define the plus/minus Selmer group.
For a finite prime $v$ of $K_{n}$, we define the local condition of the plus/minus Selmer group as  
\begin{align}
  H^{1}_{\pm}(K_{n, v}, E[p^{\infty}]) \coloneq \begin{cases}
    E(K_{n, v_{n}}) \otimes (\Q_{p}/\Z_{p}) = 0 & (v \neq p), \\
    \hat{E}(K_{n, p})^{\pm} \otimes (\Q_{p}/\Z_{p}), & (v=p).
  \end{cases}
\end{align}
Then, we define the plus/minus Selmer group of $E$ over $K_{n}$ as 
\begin{align}
  \Sel_{p^{\infty}}^{\pm}(E/K_{n}) \coloneq \Ker \left(
    H^{1}(K_{n}, E[p^{\infty}]) \rarrow \prod_{v} \frac{H^{1}(K_{n, v}, E[p^{\infty}])}{H^{1}_{\pm}(K_{n, v}, E[p^{\infty}])}
  \right)
\end{align}
and $\Sel_{p^{\infty}}^{\pm}(E/K_{\infty}) \coloneq \varinjlim_{n}\Sel_{p^{\infty}}^{\pm}(E/{K_{n}})$. 
In the rest of this chapter, we introduce the previous work for the proof of Theorem \ref{thm:mthm}.

\begin{thm}[{\cite{BBL22}}] \label{thm:cotorsion}
  Let $N_{E}$ be the conductor of $E$, and $D_{K}$ be the discriminant of $K$.
  Assume that $(D_{K}, pN_{E}) = 1$, (cp), (def) and also 
  \begin{itemize}
    \item[(Im)] the mod $p$ Galois representation $G_{\Q} \rarrow \Aut_{\F_{p}}(E[p])$ is surjective if $p=5$, and irreducible if $p > 5$,
    \item[(ram)] for all prime $\ell \mid N_{E}^{-}$ with $\ell^{2} \equiv 1 \bmod p$, the inertia subgroup $I_{\ell} \subset G_{\Q_{\ell}}$ acts non-trivially on $E[p]$.
  \end{itemize}
  Then, $\Sel_{p^{\infty}}^{\pm}(E/K_{\infty})$ is a $\Lambda$-cotorsion, and we have
  \begin{align}
    (L_{p}^{\pm}(E/K_{\infty})) \subset \ch_{\Lambda} \left(\Sel_{p^{\infty}}^{\pm}(E/K_{\infty})^{\vee} \right),
  \end{align}
  where $\ch_{\Lambda}(M)$ is the characteristic ideal for a torsion $\Lambda$-module $M$.
\end{thm}

\begin{prp}[{\cite[Corollary 1.4]{shi+}}] \label{prp:Lambda-submod}
  Consider the assumption in Theorem \ref{thm:cotorsion}.
  Then, $\Sel_{p^{\infty}}^{\pm}(E/K_{\infty})$ has no proper $\Lambda$-submodules with finite index.
\end{prp}

By the above proposition and a property of the Fitting ideal (cf. \cite[Lemma A.7]{Kim-Kur21}), we have
\begin{align}
  \Fitt_{\Lambda} \left(\Sel_{p^{\infty}}^{\pm}(E/K_{\infty})^{\vee} \right) = \ch_{\Lambda} \left(\Sel_{p^{\infty}}^{\pm}(E/K_{\infty})^{\vee} \right). \label{eq:Fitt vs char}
\end{align}

Last, we prove the control theorem for $\Sel_{p^{\infty}}^{\pm}(E/K_{\infty})$ in our situation by the similar arguments to \cite{Agb-How05}. 
We prepare the following lemma for the proof.
Let $w$ be a prime of $K_{\infty}$ above $p$.

\begin{lem}\label{lem:lemma of control}
  The natural map $f_{n}^{\pm}:H^{1}_{\pm}(K_{n, p}, E[p^{\infty}]) \rarrow H^{1}_{\pm}(K_{\infty, w}, E[p^{\infty}])[\omega_{n}^{\pm}]$ is injective, and the cokernel of $f_{n}^{\pm}$ is a finite group for any positive integer $n$.
  Here, we define $H^{1}_{\pm}(K_{\infty, w}, E[p^{\infty}]) \coloneq \varinjlim H^{1}_{\pm}(K_{n, p}, E[p^{\infty}])$.
\end{lem}
\begin{proof}
  Note that we have $E[p^{\infty}]^{G_{K_{\infty, p}}}=0$ by \cite[Proposition 3.2]{shi+}.
  By Inflation-Restriction exact sequence, we see that the canonical map
  \begin{align}
    f_{n}:H^{1}(K_{n, p}, E[p^{\infty}]) \rarrow H^{1}(K_{\infty, p}, E[p^{\infty}])^{\Gal(K_{\infty, p}/F_{n,p})}=H^{1}(K_{\infty, p}, E[p^{\infty}])[\omega_{n}]
  \end{align}
  is injective.
  Since the map $f_{n}^{\pm}$ is the restriction of $f_{n}$ to $H^{1}_{\pm}(K_{n, p}, E[p^{\infty}])$, $f_{n}^{\pm}$ is also injective.

  For the claim for $\Coker f_{n}^{\pm}$, it suffices to show that both the $\Z_{p}$-coranks of $H^{1}_{\pm}(K_{n, p}, E[p^{\infty}])$ and $H^{1}_{\pm}(K_{\infty, w}, E[p^{\infty}])[\omega_{n}^{\pm}]$ are same. 
  Since we have $H^{1}_{\pm}(K_{\infty, w}, E[p^{\infty}])^{\vee} \simeq \Lambda^{2}$ as $\Lambda$-module by the Rubin conjecture, we see that
  \begin{align}
    H^{1}_{\pm}(K_{\infty, w}, E[p^{\infty}])[\omega_{n}^{\pm}]^{\vee} &\simeq (H^{1}_{\pm}(K_{\infty, w}, E[p^{\infty}]))^{\vee}/\omega_{n}^{\pm}(H^{1}_{\pm}(K_{\infty, w}, E[p^{\infty}]))^{\vee} \\
    & \simeq \Lambda^{2}/\omega_{n}^{\pm}\Lambda^{2}.
  \end{align}
  On the other hand, we have $H^{1}_{\pm}(K_{n, p}, E[p^{\infty}]) = \hat{E}(K_{n,p})^{\pm} \otimes (\Q_{p}/\Z_{p})$ by the definition.
  Thus, we see that $H^{1}_{\pm}(K_{n, p}, E[p^{\infty}])^{\vee} \simeq \Lambda_{n}^{2}/\omega_{n}^{\pm}\Lambda_{n}^{2}$ by Proposition \ref{prp:n-th structure}.
  Therefore, the $\Z_{p}$-rank of $H^{1}_{\pm}(K_{n, p}, E[p^{\infty}])$ and $H^{1}_{\pm}(K_{\infty, w}, E[p^{\infty}])[\omega_{n}^{\pm}]$ are the same.
\end{proof}

\begin{prp} \label{prp:control thm}
  Consider the assumption in Theorem \ref{thm:cotorsion}.
  Then, the canonical homomorphism
  \begin{align}
    \Sel_{p^{\infty}}^{\pm}(E/K_{\infty})[\omega_{n}^{\pm}] \rarrow \Sel_{p^{\infty}}^{\pm}(E/K_{n})[\omega_{n}^{\pm}]
  \end{align}
  is injective, and the order of the cokernel is finite for any $n$.
\end{prp}
\begin{proof}
  We take the finite subset $\Sigma = \{ p \} \cup \{ \text{bad primes of $E$} \}$.
  Then, we have the following commutative diagram:
  \begin{align}
    \xymatrix{
      0 \ar[r] & \Sel_{p}^{\pm}(E/K_{n})[\omega_{n}^{\pm}] \ar[r] \ar[d]^{s_{n}^{\pm}} & H^{1}(K_{\Sigma}/K_{n}, E[p^{\infty}])[\omega_{n}^{\pm}] \ar[r]^-{a} \ar[d]^{h_{n}^{\pm}} & \DS \prod_{\substack{v_{n} \mid v \\ v \in \Sigma}} \frac{H^{1}(K_{n, v_{n}}, E[p^{\infty}])[\omega_{n}^{\pm}]}{H^{1}_{\pm}(K_{n, v_{n}}, E[p^{\infty}])} \ar[d]^-{g_{n}^{\pm}=\prod g_{n, v_{n}}^{\pm}}, \\
      0 \ar[r] & \Sel_{p}^{\pm}(E/K_{\infty})[\omega_{n}^{\pm}] \ar[r] & H^{1}(K_{\Sigma}/K_{\infty}, E[p^{\infty}])[\omega_{n}^{\pm}] \ar[r] &\DS \prod_{\substack{v_{\infty} \mid v \\ v \in \Sigma}} \frac{H^{1}(K_{\infty, v_{\infty}}, E[p^{\infty}])[\omega_{n}^{\pm}]}{H^{1}_{\pm}(K_{\infty, v_{\infty}}, E[p^{\infty}])[\omega_{n}^{\pm}]}.
    }
  \end{align}
  Since we have $E[p^{\infty}]^{G_{K_{\infty}}} = 0$, the map $H^{1}(K_{\Sigma}/K_{n}, E[p^{\infty}]) \rarrow H^{1}(K_{\Sigma}/K_{\infty}, E[p^{\infty}])[\omega_{n}]$ induced by the restriction map is isomorphism by the Inflation-Restriction exact sequence.
  Therefore, $h_{n}^{\pm}$ is also isomorphism.
  By the snake lemma, $s_{n}^{\pm}$ is injective, and it suffices to calculate the kernel of $g_{n}^{\pm}$.

  We first consider the case $v \in \Sigma \setminus \{ p \}$.
  Then, $K_{\infty, v_{\infty}}/K_{n, v_{n}}$ is the trivial extension or the unramified $\Z_{p}$-extension.
  If the extension $K_{\infty, v_{\infty}}/K_{n, v_{n}}$ is trivial, then it is clear that $\Ker g_{n, v_{n}}^{\pm} = 0$.
  Assume that $K_{\infty, v_{\infty}}/K_{n, v_{n}}$ is the unramified $\Z_{p}$-extension.
  Write $B_{v_{\infty}} \coloneq E[p^{\infty}]^{G_{K_{\infty, v_{\infty}}}}$.
  We consider the exact sequence
  \begin{align}
    \xymatrix{
      0 \ar[r] & H^{1}(K_{\infty, v_{\infty}}/K_{n,v_{n}}, B_{v_{\infty}}) \ar[r] & H^{1}(K_{n, v_{n}}, E[p^{\infty}]) \ar[r] & H^{1}(K_{\infty, v_{\infty}}, E[p^{\infty}])
    }
  \end{align}
  by the Inflation-Restriction exact sequence.
  Since we assume that $K_{\infty, v_{\infty}}/K_{n,v_{n}}$ is the unramified $\Z_{p}$-extension, $H^{1}(K_{\infty, v_{\infty}}/K_{n,v_{n}}, B_{v_{\infty}})$ is isomorphic to the maximum invariant quotient of $B_{v_{\infty}}$. 
  If we take a topological generator $\gamma_{v_{n}}$ of $\Gal(K_{\infty, v_{\infty}}/K_{n, v_{n}})$, then the maximal divisible subgroup $(B_{v_{\infty}})_{\mathrm{div}}$ is included in $(\gamma_{v_{n}} - 1)B_{v_{\infty}}$.
  Thus, we have
  \begin{align}
    \# \Ker(g_{n,v_{n}}^{\pm}) 
    &\leqq \# H^{1}(K_{\infty, v_{\infty}}/K_{n, v_{n}}, B_{v_{\infty}}) \\
    &= \# (B_{v_{\infty}}/(\gamma_{v_{n}} - 1)B_{v_{\infty}}) \\
    &\leqq \# (B_{v_{\infty}}/(B_{v_{\infty}})_{\mathrm{div}}),
  \end{align}
  that is, the order of $\# \Ker(g_{n,v_{n}}^{\pm})$ is finite for any $n$.

  We next consider the case $v=p$.
  Consider the commutative diagram
  \begin{align}
  \xymatrix{
    0 \ar[r] & H^{1}_{\pm}(K_{n, p}, E[p^{\infty}]) \ar[r] \ar[d]^-{f_{n}^{\pm}} & H^{1}(K_{n, p}, E[p^{\infty}])[\omega_{n}^{\pm}] \ar[r] \ar[d]^{h_{p}^{\pm}} &\DS \frac{H^{1}(K_{n, p}, E[p^{\infty}])[\omega_{n}^{\pm}]}{H^{1}_{\pm}(K_{n, p}, E[p^{\infty}])} \ar[r] \ar[d]^-{g_{n,p}^{\pm}} & 0, \\
    0 \ar[r] & H^{1}_{\pm}(K_{\infty, w}, E[p^{\infty}])[\omega_{n}^{\pm}] \ar[r] & H^{1}(F_{\infty, w}, E[p^{\infty}])[\omega_{n}^{\pm}] \ar[r] &\DS \frac{H^{1}(F_{\infty, w},  E[p^{\infty}])[\omega_{n}^{\pm}]}{H^{1}_{\pm}(K_{\infty, w}, E[p^{\infty}])[\omega_{n}^{\pm}]} \ar[r] & 0.
  }
  \end{align}
  Since $h_{p}^{\pm}$ is isomorphism as Lemma \ref{lem:lemma of control}, we have $\Ker(g_{n,p}^{\pm}) =\Coker(f_{n}^{\pm})$.
  However, $\# \Coker(f_{n}^{\pm})$ is finite by Lemma \ref{lem:lemma of control}.
  Therefore, the order of $\Ker(g_{n,p}^{\pm})$ is also finite for any $n$.
\end{proof}

\section{The proof of Theorem \ref{thm:mthm}}
\noindent
In this section, we complete to prove Theorem \ref{thm:mthm}.
First, the following proposition can be proved from a consequence of previous works.

\begin{prp}\label{prp:conclusion}
  We assume (cp), (def), (Im), and (ram).
  Then, we have 
  \begin{align}
    \left( \omega_{n}^{\mp} \cdot L_{p}^{\pm}(E/K_{\infty}) \bmod \omega_{n} \right) \subset (\omega_{n}^{\mp}) \cdot \Fitt_{\Lambda_{n}}\left( \Sel_{p^{\infty}}^{\pm}(E/K_{n})^{\vee} \right)
  \end{align}
  as an ideal of $\Lambda_{n}$.
\end{prp}
\begin{proof}
  By applying Theorem \ref{thm:cotorsion} and \eqref{eq:Fitt vs char}, we see that
  \begin{align}
    \left( L_{p}^{\pm}(E/K_{\infty}) \right) \subset \Fitt_{\Lambda}(\Sel_{p^{\infty}}^{\pm}(E/K_{\infty})^{\vee}).
  \end{align} 
  Taking the quotient by $\omega_{n}^{\pm}$, we have
  \begin{align}
    \left( L_{p}^{\pm}(E/K_{\infty}) \bmod \omega_{n}^{\pm} \right) 
    \subset \Fitt_{\Lambda_{n}/\omega_{n}^{\pm}\Lambda_{n}} \left( (\Sel_{p^{\infty}}^{\pm}(E/K_{\infty})[\omega_{n}^{\pm}])^{\vee} \right)
  \end{align}
  as ideals of $\Lambda_{n}/\omega_{n}^{\pm}\Lambda_{n}$, respectively.
  By Proposition \ref{prp:control thm}, we see that
  \begin{align}
    \left( L_{p}^{\pm}(E/K_{\infty}) \bmod \omega_{n}^{\pm} \right) \subset \Fitt_{\Lambda_{n}/\omega_{n}^{\pm}\Lambda_{n}} \left( (\Sel_{p^{\infty}}^{\pm}(E/K_{n})[\omega_{n}^{\pm}])^{\vee} \right)
  \end{align}
  as ideals of $\Lambda_{n}/\omega_{n}^{\pm}\Lambda_{n}$, respectively.
  By the base change from $\Lambda_{n}/\omega_{n}^{\pm}\Lambda_{n}$ to $\Lambda_{n}$,  we have
  \begin{align}
    \Fitt_{\Lambda_{n}/\omega_{n}^{\pm}\Lambda_{n}} \left( (\Sel_{p^{\infty}}^{\pm}(E/K_{n})[\omega_{n}^{\pm}])^{\vee} \right)
    =\frac{\Fitt_{\Lambda_{n}}(\Sel_{p^{\infty}}^{\pm}(E/K_{n})^{\vee}) + (\omega_{n}^{\pm})}{(\omega_{n}^{\pm})}.
  \end{align}
  as ideals of $\Lambda_{n}/\omega_{n}^{\pm}\Lambda_{n}$, respectively.
  Therefore, we have 
  \begin{align}
    \left( L_{p}^{\pm}(E/K_{\infty}) \bmod \omega_{n}^{\pm} \right) + (\omega_{n}^{\pm}) 
    \subset \Fitt_{\Lambda_{n}}(\Sel_{p^{\infty}}^{\pm}(E/K_{n})^{\vee}) + (\omega_{n}^{\pm}).
  \end{align}
  as ideals of $\Lambda_{n}$.
  By multiplying $\omega_{n}^{\mp}$, the claim is proved.
\end{proof}

From now, we only focus on the sign of 
\begin{align}
  \left( \omega_{n}^{+} \cdot L_{p}^{-}(E/K_{\infty}) \bmod \omega_{n} \right) \subset (\omega_{n}^{+}) \cdot \Fitt_{\Lambda_{n}}\left( \Sel_{p^{\infty}}^{-}(E/K_{n})^{\vee} \right). \label{eq:conclusion}
\end{align}
By the definition of the (plus/minus) Selmer group, we have the following exact sequence:
\begin{align}
  \left( \frac{\hat{E}(K_{n, p}) \otimes (\Q_{p}/\Z_{p})}{\hat{E}(K_{n, p})^{-} \otimes (\Q_{p}/\Z_{p})} \right)^{\vee} \overset{\iota^{-}}{\rarrow} \Sel_{p^{\infty}}(E/K_{n})^{\vee} \rarrow \Sel_{p^{\infty}}^{-}(E/K_{n})^{\vee} \rarrow 0.
\end{align}
The properties of the Fitting ideal for an exact sequence and a surjective homomorphism imply the inclusions

{\small
\begin{align}
  \Fitt_{\Lambda_{n}}\left(\left( \frac{\hat{E}(K_{n, p}) \otimes (\Q_{p}/\Z_{p})}{\hat{E}(K_{n, p})^{-} \otimes (\Q_{p}/\Z_{p})} \right)^{\vee} \middle/ \Ker \iota^{-}\right) \cdot \Fitt_{\Lambda_{n}}\left( \Sel_{p^{\infty}}^{-}(E/K_{n})^{\vee} \right) \subset \Fitt_{\Lambda_{n}}\left( \Sel_{p^{\infty}}(E/K_{n})^{\vee} \right), \\
  \Fitt_{\Lambda_{n}} \left( \frac{\hat{E}(K_{n, p}) \otimes (\Q_{p}/\Z_{p})}{\hat{E}(K_{n, p})^{-} \otimes (\Q_{p}/\Z_{p})} \right)^{\vee} \subset \Fitt_{\Lambda_{n}} \left(\left( \frac{\hat{E}(K_{n, p}) \otimes (\Q_{p}/\Z_{p})}{\hat{E}(K_{n, p})^{-} \otimes (\Q_{p}/\Z_{p})} \right)^{\vee} \middle/ \Ker \iota^{-}\right). 
\end{align}
}
By Theorem \ref{thm:decomposition} and Proposition \ref{prp:n-th structure}, we can compute
\begin{align}
  \left(\frac{\hat{E}(K_{n, p}) \otimes (\Q_{p}/\Z_{p})}{\hat{E}(K_{n, p})^{-} \otimes (\Q_{p}/\Z_{p})} \right)^{\vee} = \left( \hat{E}(K_{n, p})^{+} \otimes (\Q_{p}/\Z_{p}) \right)^{\vee} \simeq \left( \Lambda_{n}/\omega_{n}^{+} \Lambda_{n} \right)^{\oplus 2}. 
\end{align}
Hence, we see that
\begin{align}
  \Fitt_{\Lambda_{n}} \left( \frac{\hat{E}(K_{n, p}) \otimes (\Q_{p}/\Z_{p})}{\hat{E}(K_{n, p})^{-} \otimes (\Q_{p}/\Z_{p})} \right)^{\vee} &= \Fitt_{\Lambda_{n}}\left( \left( \Lambda_{n}/\omega_{n}^{+}\Lambda_{n} \right)^{\oplus 2} \right) =(\omega_{n}^{+})^{2}
\end{align}
From the above calculations, we have
\begin{align}
  &(\omega_{n}^{+})^{2} \cdot \Fitt_{\Lambda_{n}}(\Sel_{p^{\infty}}^{-}(E/K_{n})^{\vee}) \\ 
  &= \Fitt_{\Lambda_{n}} \left( \frac{\hat{E}(K_{n, p}) \otimes (\Q_{p}/\Z_{p})}{\hat{E}(K_{n, p})^{-} \otimes (\Q_{p}/\Z_{p})} \right)^{\vee} \cdot \Fitt_{\Lambda_{n}}(\Sel_{p^{\infty}}^{-}(E/K_{n})^{\vee}) \\
  &\subset \Fitt_{\Lambda_{n}}\left(\left( \frac{\hat{E}(K_{n, p}) \otimes (\Q_{p}/\Z_{p})}{\hat{E}(K_{n, p})^{-} \otimes (\Q_{p}/\Z_{p})} \right)^{\vee} \middle/ \Ker \iota^{-}\right) \cdot \Fitt_{\Lambda_{n}}\left( \Sel_{p^{\infty}}^{-}(E/K_{n}) \right) \\
  & \subset \Fitt_{\Lambda_{n}} (\Sel_{p^{\infty}}(E/K_{n})^{\vee}).
\end{align}
Multiplying both sides of the inclusion \eqref{eq:conclusion} by $\omega_{n}^{+}$, we see that
\begin{align}
  \left( (\omega_{n}^{+})^{2} \cdot L_{p}^{-}(E/K_{\infty}) \bmod \omega_{n}\right) &\subset (\omega_{n}^{+})^{2} \cdot \Fitt_{\Lambda_{n}} \left( \Sel_{p^{\infty}}^{-}(E/K_{n})^{\vee} \right) \\
  &\subset \Fitt_{\Lambda_{n}} \left( \Sel_{p^{\infty}}(E/K_{n})^{\vee} \right). \label{eq:Selmer}
\end{align}
On the other hand, we have 
\begin{align}
  \left( (\omega_{n}^{+})^{2} \cdot L_{p}^{-}(E/K_{\infty}) \bmod \omega_{n}\right) &= \left( (\omega_{n}^{+})^{2} \cal{L}_{f}^{-} \iota(\cal{L}_{f}^{-}) \bmod \omega_{n} \right) \\
  &= \left( (\omega_{n}^{+})^{2} \cal{L}_{n}^{-} \cdot \iota (\cal{L}_{n}^{-}) \right) \\
  &= (\cal{L}_{f, n} \cdot \iota(\cal{L}_{f, n})) \\
  &= (L_{p}(E/K_{n})).
  \label{eq:pL}
\end{align}
Combining \eqref{eq:Selmer} and \eqref{eq:pL}, the proof of Theorem \ref{thm:mthm} is completed.

\bibliographystyle{plain}
\bibliography{cite}
\end{document}